\documentclass[11pt,reqno]{amsart}
\usepackage{amsmath,amsthm,amssymb,cite}

\numberwithin{equation}{section}
\newtheorem{Def}{Definition}[section]

\newtheorem{thm}[Def]{Theorem}
\newtheorem{cor}[Def]{Corollary}

\theoremstyle{remark}
\newtheorem{rmk}[Def]{Remark}
\makeatletter
\@namedef{subjclassname@2010}{%
  \textup{2010} Mathematics Subject Classification}
\makeatother

\allowdisplaybreaks

\begin{document}
 \title[ On Coefficient problem for  bi-univalent analytic functions]{On Coefficient problem for bi-univalent analytic functions\boldmath}
\author[N. Bohra]{Nisha Bohra}

\address{Department of Mathematics, University of Delhi,
Delhi--110 007, India}
\email{nishib89@gmail.com}

\author[V. Ravichandran]{V. Ravichandran}

\address{Department of Mathematics, University of Delhi,
Delhi--110 007, India}

\email{vravi@maths.du.ac.in, vravi68@gmail.com}

\begin{abstract}
 Estimates for initial coefficients  of Taylor-Maclaurin series of  bi-univalent functions belonging to certain classes defined by subordination are obtained. Our estimates improve upon the earlier known estimates for second and third coefficient. The bound for the fourth coefficient is new. In addition, bound for the fifth coefficient is obtained for  bi-starlike and strongly bi-starlike functions of order $\rho$ and $\beta$ respectively.
\end{abstract}



 \maketitle
\section{Introduction}
Let $\mathcal{A}$ be the class of analytic functions $f$ in the open unit disk $\mathbb{D}=\lbrace z\in \mathbb{C} : \vert z\vert <1\rbrace$ and normalized by the conditions $f(0)=0$ and $f'(0)=1$. If $f\in \mathcal{A}$, then
\begin{equation}\label{aa}
 f(z)=z+\sum_{n=2}^{\infty}a_n z^n  \quad (z\in\mathbb{D}).
 \end{equation}
The Koebe one-quarter theorem assures that the image of unit disk  $\mathbb{D}$ under every univalent function $f\in \mathcal{A}$ contains a disk of radius 1/4. Thus every univalent function $f$ has an inverse $f^{-1}$ satisfying $f^{-1}(f(z))=z$ $(z\in \mathbb{D})$ and
\[f(f^{-1}(w))=w  \quad (\vert w\vert < r_{0}(f), r_{0}(f)\geq 1/4).\]
  Furthermore, the Tayor-Maclaurin series of $f^{-1}$ is given by
  \begin{equation}\label{29}
  \begin{split}
  f^{-1}(w)=& w-a_2w^2+(2a_2^2-a_3)w^3-(5a_2^3-5a_2a_3+a_4)w^4\\&+(14a_2^4-21a_2^2a_3+3a_3^2+6a_2a_4-a_5)w^5+\cdots.
  \end{split}
  \end{equation}
  A function $f\in \mathcal{A}$ is said to be bi-univalent in $\mathbb{D}$ if  $f$ is univalent and $f^{-1}$ has univalent analytic continuation, which we denote by $g$, to the unit disk   $\mathbb{D}$. Let $\sigma$ denote the class of bi-univalent functions defined in the unit disk $\mathbb{D}$. Coefficient problem for bi-univalent functions were recently investigated by several authors \cite{bulut,frasin1,frasin,sim,siva,tang,xu,zap}.  An analytic function $f$ is subordinate to an analytic function $g$, written as $f(z)\prec g(z)$, provided there is an analytic function $w$ defined on $\mathbb{D}$ with $w(0)=0$ and $\vert w(z)\vert <1$ satisfying $f(z)=g(w(z))$. Ma and Minda unified various subclasses of starlike  and convex functions for which either of the quantity $zf'(z)/f(z)$ or $1+zf''(z)/f'(z)$ is subordinate to a more general superordinate function. For this purpose, they considered an analytic function $\varphi$ with positive real part in the unit disk $\mathbb{D}$ and normalized by  $\varphi(0)=1$ and $\varphi'(0)>0$. The class of Ma-Minda starlike functions consists of functions $f\in \mathcal{A}$ satisfying the subordination $ zf'(z)/f(z)\prec \varphi(z) $. Similarly, the class of Ma-Minda  convex functions  consists of functions $f\in \mathcal{A} $ satisfying the subordination $1+zf''(z)/f'(z)\prec \varphi(z)$. A function $f$ is bi-starlike of Ma-Minda type or bi-convex of Ma-Minda type if both $f$ and $g$  are respectively Ma-Minda starlike or convex.
 The classes consisting of  bi-starlike of Ma-Minda type or bi-convex of Ma-Minda type functions  are denoted by $ST_{\sigma}(\varphi)$ and $CV_{\sigma}(\varphi)$ respectively.

 In this paper, we consider  more general classes $ST_{\sigma} ^{\lambda}(\varphi)$  and $M_{\sigma}^{\lambda}(\varphi)$ for  $\lambda\geq 0$ which were investigated by Ali\emph{ et al.} in \cite{ali}  wherein they obtained the  bounds for $a_2$ and $a_3$. This motivated us to improve the bounds for $a_2$ and $a_3$. We also find the bound for  $a_4$. Earlier for $ 0\leq \rho<1$ and $0<\beta\leq 1$, Brannan and Taha \cite{bt} introduced two interesting subclasses $ST_{\sigma}(\rho)\equiv ST_{\sigma}((1+(1-2\rho)z)/(1-z))$  and $SS_{\sigma}(\beta)\equiv ST_{\sigma}(((1+z)/(1-z))^{\beta})$  of the  class $\sigma$, in analogy to the subclasses of  starlike functions of order $\rho$ and strongly starlike functions of order $\beta$ of the class $\mathcal{A}$ respectively. They found estimates for the second and third Taylor-Maclaurin coefficients  of the functions $f$  in these classes. Recently Mishra and Soren \cite{ak} found estimates for the fourth Taylor-Maclaurin coefficients  of the functions $f$  in these classes. This motivated us to find the bound for the fifth coefficient.

\section{Coefficient estimates}
Throughout this paper $\varphi$ denotes an analytic univalent function in $\mathbb{D}$ with positive real part and normalized by  $\varphi(0)=1$ and $\varphi'(0)>0$. Such a function has series expansion of the form
\begin{equation}\label{ee}
 \varphi(z)=1+B_1z+B_2z^2+B_3z^3+\cdots \quad (B_1>0).
\end{equation}
\begin{Def}\label{75}
For $\lambda\geq 0$,  the class $ST ^{\lambda}(\varphi)$ consists of functions $f\in \mathcal{A}$ satisfying
\[ \frac{zf'(z)}{ f(z)}+\lambda\frac {z^2 f''(z)}{f(z)}\prec \varphi(z) \quad (z\in\mathbb{D}).\] The class $ST_{\sigma} ^{\lambda}(\varphi)$ consists of functions $f\in \sigma$ such that $f$ and $g\in ST^{\lambda}(\varphi)$ where $g$ is the analytic continuation of $f^{-1}$ to the unit disk $\mathbb{D}$.
 \end{Def}
 Note that $ST^{0}_{\sigma}(\varphi) \equiv ST_{\sigma}(\varphi)$ is the class of Ma-Minda bi-starlike functions.
The class  $ST_{\sigma}(\varphi)$ includes many earlier classes, which are mentioned below:
\begin{enumerate}
\item  For $-1\leq B<A\leq 1$, $ST_{\sigma}((1+Az)/(1+Bz))\equiv ST_{\sigma}[A,B]$ and, for $0\leq \rho <1$,  $ST_{\sigma}[1-2\rho, 1]\equiv ST_{\sigma}(\rho)$ is the class of bi-starlike functions of order $\rho$ introduced and studied in \cite{bt}.
\item For $0<\beta\leq 1$, $ST_{\sigma}(((1+z)/(1-z))^{\beta})\equiv SS_{\sigma}(\beta)$ $(0<\beta\leq 1)$ is the class of strongly bi- starlike functions of order $\beta$ introduced and studied in \cite{bt}.
\end{enumerate}

\begin{thm}\label{vv}
Let the function $f$ given by (\ref{aa})  be in the class  $ST_{\sigma} ^{\lambda}(\phi)$ for  $\lambda\geq 0 $.
\begin{itemize}
\item[(a)]If $(1+2\lambda)^2B_1\leq \ \vert  (1+4\lambda)B_1^2+(B_1-B_2)(1+2\lambda)^2\vert$, then\\
\text{the  coefficients $a_2$, $a_3$ and $a_4$ satisfies }
\begin{align*}
\vert a_2\vert &\leq
\dfrac{B_1\sqrt{B_1}}{\sqrt{\vert(1+4\lambda)B_1^2+(B_1-B_2)(1+2\lambda)^2\vert}},\\
\vert a_3\vert  &\leq  \min \bigg\lbrace \dfrac{B_1}{ 4(1+3\lambda)\vert(1+4\lambda)B_1^2+(B_1-B_2)(1+2\lambda)^2\vert}\\&\quad \bigg(\vert (3+10\lambda)B_1^2+(B_1-B_2)(1+2\lambda)^2\vert+\vert(1+2\lambda)B_1^2-(1+2\lambda)^2(B_1-B_2)\vert\bigg),\\&\quad \dfrac{B_1}{2(1+3\lambda)\vert(1+4\lambda)B_1^2+(B_1-B_2)(1+2\lambda)^2)\vert}\\&\quad
 \bigg(\vert(1+2\lambda)B_1^2-(1+2\lambda)^2(B_1-B_2)\vert+\vert(1+4\lambda)B_1^2+(B_1-B_2)
 (1+2\lambda)^2\vert\bigg)\bigg\rbrace, \\
 \vert a_4\vert &\leq \min \bigg\lbrace
\dfrac{B_1}{3(1+4\lambda)}+\dfrac {2\sqrt{B_1}(1+2\lambda)}{3(1+4\lambda)\sqrt{\vert(1+4\lambda)B_1^2+
(B_1-B_2)(1+2\lambda)^2\vert}}(\vert A\vert+\vert C\vert),\\&\quad\dfrac{2 B_1}{ 6(1+4\lambda)\vert
 (9+44\lambda)B_1^2 -8(1+2\lambda)(1+3\lambda)(B_2 -B_1)\vert}\bigg(\vert
 (12+52\lambda)B_1^2\\&\quad -4(1+2\lambda)(1+3\lambda)(B_2-B_1)\vert + \vert(3+8\lambda)B_1^2+4(1+2\lambda)
 (1+3\lambda)(B_2-B_1)\vert\bigg) \\&\quad +\dfrac{2(1+2\lambda)\sqrt{B_1}}
{3(1+4\lambda)\sqrt{\vert(1+4\lambda)B_1^2+(B_1-B_2)(1+2\lambda)^2\vert}}\\&\quad\bigg\vert (B_2-B_1)+
\dfrac{B_1(2(1+3\lambda)B_1^3+(1+2\lambda)^3(B_1+B_3-2B_2))}{2(1+2\lambda)((1+4\lambda)B_1^2+(B_1-B_2)
(1+2\lambda)^2)}\bigg\vert\bigg \}.
 \end{align*}

\item[(b)] If $(1+2\lambda)^2B_1\geq \ \vert  (1+4\lambda)B_1^2+(B_1-B_2)(1+2\lambda)^2\vert$, then\\
\text{the  coefficients $a_2$, $a_3$ and $a_4$ satisfies }
\begin{align*}
\vert a_2\vert &\leq
\dfrac{B_1}{1+2\lambda},\\
\vert a_3\vert& \leq  \min \bigg\lbrace \dfrac{B_1}{ 4(1+3\lambda)\vert(1+4\lambda)B_1^2+(B_1-B_2)(1+2\lambda)^2\vert}\\&\quad\bigg(\vert (3+10\lambda)B_1^2+(B_1-B_2)(1+2\lambda)^2\vert+\vert(1+2\lambda)B_1^2-(1+2\lambda)^2(B_1-B_2)\vert\bigg),\\& \quad\dfrac{1}{2(1+3\lambda)(1+2\lambda)}\big(\vert B_1^2-(1+2\lambda)(B_1-B_2)\vert+ (1+2\lambda)B_1\big)\bigg\rbrace,\\
\vert a_4\vert& \leq \min \bigg\lbrace\dfrac{B_1}{3(1+4\lambda)}+\dfrac{2}{3(1+4\lambda)}(\vert A\vert +\vert C\vert),\\&\quad \dfrac{2 B_1}{6(1+44\lambda)\vert (9+44\lambda)B_1^2-8(1+2\lambda)(1+3\lambda)(B_2-B_1)\vert}\\&\quad\bigg(\vert (12+52\lambda)B_1^2 -4(1+2\lambda)(1+3\lambda)(B_2-B_1)\vert\\&\quad +\vert(3+8\lambda)B_1^2+4(1+2\lambda)(1+3\lambda)(B_2-B_1)\vert\bigg)\\&\quad +\frac{2}{3(1+4\lambda)}\bigg\vert (B_2-B_1) +\dfrac{B_1(2(1+3\lambda)B_1^3+(1+2\lambda)^3(B_1+B_3-2B_2))}{2(1+2\lambda)((1+4\lambda)B_1^2+(B_1-B_2)(1+2\lambda)^2)}\bigg\vert \bigg\}.
\end{align*}
where
\begin{equation*}
\begin{split}
A=&(B_2-B_1)+\frac{(3+8\lambda)B_1^2}{8(1+2\lambda)(1+3\lambda)}+\frac{B_1(1+2\lambda)^2}{4((1+4\lambda)B_1^2+(B_1-B_2)(1+2\lambda)^2)}\\&\quad\bigg(\frac{2(1+3\lambda)B_1^3}{(1+2\lambda)^3}+(B_1+B_3-2B_2)\bigg)
\end{split}
\end{equation*}
and
\begin{equation*}
\begin{split}
C=&-\frac{(3+8\lambda)B_1^2}{8(1+2\lambda)(1+3\lambda)}+\frac{B_1(1+2\lambda)^2}{4((1+4\lambda)B_1^2+(B_1-B_2)(1+2\lambda)^2)}\\&\quad\bigg(\frac{2(1+3\lambda)B_1^3}{(1+2\lambda)^3}+(B_1+B_3-2B_2)\bigg).
\end{split}
\end{equation*}
\end{itemize}
\end{thm}
\begin{proof}
Since $f\in ST_{\sigma} ^{\lambda}(\phi)$, there exists two analytic functions $r,s:\mathbb{D}\rightarrow \mathbb{D}$, with $r(0)=0=s(0)$, such that
\begin{equation}\label{bbb}
 \frac{zf'(z)}{ f(z)}+\lambda\frac {z^2 f''(z)}{f(z)}= \varphi(r(z))\quad \textsl{and} \quad \frac{wg'(w)}{ g(w)}+\lambda\frac {w^2 g''(w)}{g(w)}= \varphi(s(w))
\end{equation}
Define the functions $p$ and $q$ by
\begin{equation*}
p(z)=\frac{1+r(z)}{1-r(z)}=1+p_1z+p_2z^2+p_3z^3+p_4z^4+\cdots
\end{equation*}
and
\begin{equation*}
q(w)=\frac{1+s(w)}{1-s(w)}=1+q_1w+q_2w^2+q_3w^3+q_4w^4+\cdots
\end{equation*}
or equivalently,
\begin{equation}\label{cc}
r(z)=\frac{p(z)-1}{p(z)+1}=\frac{1}{2}\bigg(p_1z+\bigg(p_2-\frac{p_1^2}{2}\bigg)z^2+\frac{1}{4}(p_1^3-4p_1p_2+4p_3)z^3+\cdots\bigg)
\end{equation}
and
\begin{equation}\label{dd}
s(w)=\frac{q(w)-1}{q(w)+1}=\frac{1}{2}\bigg(q_1w+\bigg(q_2-\frac{q_1^2}{2}\bigg)w^2+\frac{1}{4}(q_1^3-4q_1q_2+4q_3)w^3+\cdots\bigg).
\end{equation}
Then $p$ and $q$ are analytic in $\mathbb{D}$ with $p(0)=1=q(0)$. Since $r,s:\mathbb{D}\rightarrow \mathbb{D}$, the functions $p$ and $q$ have positive real part in $\mathbb{D}$, and hence $\vert p_i\vert\leq 2$ and $\vert q_i\vert\leq 2$.
Using (\ref{bbb}), (\ref{cc}) and (\ref{dd}), we have
\begin{equation}\label{jj}
 \frac{zf'(z)}{ f(z)}+\lambda\frac {z^2 f''(z)}{f(z)}= \varphi\bigg(\frac{p(z)-1}{p(z)+1}\bigg)\quad \textsl{and}\quad \frac{wg'(w)}{ g(w)}+\lambda\frac {w^2 g''(w)}{g(w)}= \varphi\bigg(\frac{q(w)-1}{q(w)+1}\bigg).
\end{equation}
Using (\ref{ee}) with (\ref{cc}) and (\ref{dd}), it is evident that
\begin{equation}\label{ff}
\begin{split}
\varphi\bigg(\frac{p(z)-1}{p(z)+1}\bigg)&=1+\frac{1}{2}B_1p_1z+\bigg(\frac{1}{2}B_1\bigg(p_2-\frac{1}{2}
p_1^2\bigg)+\frac{1}{4}B_2p_1^2\bigg)z^2  \\
           &\quad+\bigg(\frac{1}{2}B_1\bigg(\frac{p_1^3}{4}-p_1p_2+p_3\bigg)+\frac{1}{2}
B_2p_1\bigg(p_2-\frac{p_1^2}{2}\bigg)+\frac{1}{8}B_3p_1^3\bigg)z^3+\cdots
\end{split}
\end{equation}
and similarly
\begin{equation}\label{gg}
\begin{split}
\varphi\bigg(\frac{q(w)-1}{q(w)+1}\bigg)&=1+\frac{1}{2}B_1q_1w+\bigg(\frac{1}{2}B_1\bigg(q_2-\frac{1}{2}
q_1^2\bigg)+\frac{1}{4}B_2q_1^2\bigg)w^2  \\
           &\quad+\bigg(\frac{1}{2}B_1\bigg(\frac{q_1^3}{4}-q_1q_2+q_3\bigg)+\frac{1}{2}
B_2q_1\bigg(q_2-\frac{q_1^2}{2}\bigg)+\frac{1}{8}B_3q_1^3\bigg)w^3+\cdots.
\end{split}
\end{equation}
Also, using (\ref{aa}), we get
\begin{equation}\label{hh}
\begin{split}
\frac{zf'(z)}{ f(z)}+\lambda\frac {z^2 f''(z)}{f(z)}&=1+(1+2\lambda)a_2z+(2(1+3\lambda)a_3-(1+2\lambda)a_2^2)z^2\\
&\quad+(3(1+4\lambda)a_4-(3+8\lambda)a_2a_3+(1+2\lambda)a_2^3)z^3+\cdots
\end{split}
\end{equation}
and using (\ref{29}), we get
\begin{equation}\label{ii}
\begin{split}
\frac{wg'(w)}{ g(w)}+\lambda\frac {w^2 g''(w)}{g(w)}&=1-(1+2\lambda)a_2w+(-2(1+3\lambda)a_3+(3+10\lambda)a_2^2)w^2\\&\quad +(-3(1+4\lambda)a_4+(12+52\lambda)a_2a_3-(10+46\lambda)a_2^3)w^3+\cdots.
\end{split}
\end{equation}
On equating coefficients of both sides of (\ref{jj}) using (\ref{ff}), (\ref{gg}), (\ref{hh}) and (\ref{ii}), we get
\begin{align}
(1+2\lambda)a_2&=\frac{1}{2}B_1p_1\label{j}\\
2(1+3\lambda)a_3-(1+2\lambda)a_2^2 &=\frac{1}{2}B_1\bigg(p_2-\frac{1}{2} p_1^2\bigg)+\frac{1}{4}B_2p_1^2\label{kk}\\
3(1+4\lambda)a_4-(3+8\lambda)a_2a_3+(1+2\lambda)a_2^3 &=\frac{1}{2}B_1\bigg(\frac{p_1^3}{4}-p_1p_2+p_3\bigg)+\frac{1}{2}B_2p_1\bigg(p_2-\frac{p_1^2}{2}\bigg)\notag\\&\quad+\frac{1}{8}B_3p_1^3\label{ll}\\
-(1+2\lambda)a_2&=\frac{1}{2}B_1q_1\label{mm}\\
-2(1+3\lambda)a_3+(3+10\lambda)a_2^2 &=\frac{1}{2}B_1\bigg(q_2-\frac{1}{2} q_1^2\bigg)+\frac{1}{4}B_2q_1^2\label{nn}\\
-3(1+4\lambda)a_4+(12+52\lambda)a_2a_3-(10+46\lambda)a_2^3 &=\frac{1}{2}B_1\bigg(\frac{q_1^3}{4}-q_1q_2+q_3\bigg)+\frac{1}{2} B_2q_1\bigg(q_2-\frac{q_1^2}{2}\bigg)\notag\\&\quad+\frac{1}{8}B_3q_1^3\label{68}.
\end{align}
From (\ref{j}) and (\ref{mm}), we see that  $p_1=-q_1$. Adding equations (\ref{kk}) and (\ref{nn}), we get
\begin{gather}
(2+8\lambda)a_2^2 = \frac{1}{2}B_1\bigg(p_2+q_2-\frac{1}{2}(p_1^2+q_1^2)\bigg)+\frac{1}{4}B_2(p_1^2+q_1^2)\nonumber\\
                \quad\quad = \frac{1}{2}B_1(p_2+q_2)-\frac{1}{4}(p_1^2+q_1^2)(B_1-B_2).\label{45}
\end{gather}
Using $p_1=-q_1$ and $a_2=B_1p_1/ 2(1+2\lambda)$ in (\ref{45}), we get
\begin{equation}\label{pp}
p_1^2=\dfrac{(p_2+q_2)B_1(1+2\lambda)^2}{(1+4\lambda)B_1^2+(B_1-B_2)(1+2\lambda)^2}.
\end{equation}
Using $\vert p_i\vert\leq 2$ and $\vert q_i\vert\leq 2$ in (\ref{pp}), we have the following refined estimate for $\vert p_1\vert$:
\begin{equation}\label{ppp}
\vert p_1\vert\leq
\begin{cases}
\dfrac{2(1+2\lambda)\sqrt{B_1}}{\sqrt{\vert(1+4\lambda)B_1^2+(B_1-B_2)(1+2\lambda)^2\vert}} \quad \textsl{if}\quad (1+2\lambda)^2B_1\leq \vert  (1+4\lambda)B_1^2\\\quad\quad\quad\quad\quad\quad\quad\quad\quad\quad\quad\quad\quad\quad\quad\quad\quad\quad\quad+(B_1-B_2)(1+2\lambda)^2\vert,\\
2 \quad \textsl{if} \quad (1+2\lambda)^2B_1\geq \vert  (1+4\lambda)B_1^2+(B_1-B_2)(1+2\lambda)^2\vert.
\end{cases}
\end{equation}
Also from (\ref{pp}), we have \[p_2+q_2=\dfrac{((1+4\lambda)B_1^2+(B_1-B_2)(1+2\lambda)^2) p_1^2}{B_1(1+2\lambda)^2}.\] This gives
\begin{equation}\label{ww}
\vert p_2+q_2\vert \leq
\begin{cases}
4 \quad \textsl{if}\quad  B_1(1+2\lambda)^2\leq \vert (1+4\lambda)B_1^2+(B_1-B_2)(1+2\lambda)^2\vert,\\
\dfrac{4\vert (1+4\lambda)B_1^2+(B_1-B_2)(1+2\lambda)^2\vert}{B_1(1+2\lambda)^2}\quad \textsl{if} \quad \ B_1(1+2\lambda)^2\geq \vert (1+4\lambda)B_1^2\\\quad\quad\quad\quad\quad\quad\quad\quad\quad\quad\quad\quad\quad\quad\quad\quad\quad\quad\quad+(B_1-B_2)(1+2\lambda)^2\vert.
\end{cases}
\end{equation}
Now using (\ref{ppp}) in (\ref{j}), we have the desired estimate for $\vert a_2\vert$:
\begin{equation}
\vert a_2\vert \leq
\begin{cases}
\dfrac{B_1\sqrt{B_1}}{\sqrt{\vert(1+4\lambda)B_1^2+(B_1-B_2)(1+2\lambda)^2\vert}}  \quad \textsl{if}\quad (1+2\lambda)^2B_1\leq \vert  (1+4\lambda)B_1^2\\\quad\quad\quad\quad\quad\quad\quad\quad\quad\quad\quad\quad\quad\quad\quad\quad\quad\quad\quad+(B_1-B_2)(1+2\lambda)^2\vert,\\
\dfrac{B_1}{1+2\lambda} \quad \textsl{if} \quad (1+2\lambda)^2B_1\geq \vert  (1+4\lambda)B_1^2+(B_1-B_2)(1+2\lambda)^2\vert.
\end{cases}
\end{equation}
To find an estimate for $\vert a_3\vert$, we express $a_3$ in terms of $p_i$'s and $q_i$'s. Subtracting (\ref{nn}) from (\ref{kk}) and using $p_1=-q_1$, we get
\begin{equation}\label{qq}
4(1+3\lambda)a_3=4(1+3\lambda)a_2^2+\dfrac{B_1}{2}(p_2-q_2).
\end{equation}
Substituting  the value of $a_2$ from (\ref{j}) and then using (\ref{pp}), we get
\begin{align}
4(1+3\lambda)a_3 =&\dfrac{(1+3\lambda)B_1^3(p_2+q_2)}{(1+4\lambda)B_1^2+(B_1-B_2)(1+2\lambda)^2}+
\dfrac{B_1}{2}(p_2-q_2)\nonumber\\
=&\dfrac{B_1}{2((1+4\lambda)B_1^2+(B_1-B_2)(1+2\lambda)^2)}\big((3+10\lambda)B_1^2+(B_1-B_2)
(1+2\lambda)^2)p_2\nonumber\\&+((1+2\lambda)B_1^2-(1+2\lambda)^2(B_1-B_2))q_2\big)\label{23}\\
=&\dfrac{B_1}{2((1+4\lambda)B_1^2+(B_1-B_2)(1+2\lambda)^2)}\big((1+2\lambda)B_1^2-(1+2\lambda)^2(B_1-B_2))(p_2+q_2)\nonumber\\
&+2((1+4\lambda)B_1^2+(B_1-B_2)(1+2\lambda)^2)p_2\big)\label{24}.
\end{align}
Using the estimates  $\vert p_2\vert \leq 2$ and $\vert q_2\vert \leq 2$ in (\ref{23}), we get
\begin{equation}\label{25}
\begin{split}
\vert a_3\vert \leq & \dfrac{B_1}{ 4(1+3\lambda)\vert(1+4\lambda)B_1^2+(B_1-B_2)(1+2\lambda)^2\vert}\bigg(\vert (3+10\lambda)B_1^2+(B_1-B_2)(1+2\lambda)^2\vert\\&+\vert(1+2\lambda)B_1^2-(1+2\lambda)^2(B_1-B_2)\vert\bigg)
\end{split}
\end{equation}
Using  equation (\ref{24}), (\ref{ww}) and $\vert p_2\vert\leq 2$, we have another estimate for $\vert a_3\vert$
\begin{equation}\label{26}
\vert a_3\vert\leq
\begin{cases}
 \dfrac{B_1(\vert(1+2\lambda)B_1^2-(1+2\lambda)^2(B_1-B_2)\vert+\vert(1+4\lambda)B_1^2+(B_1-B_2) (1+2\lambda)^2\vert)}{2(1+3\lambda)\vert(1+4\lambda)B_1^2+(B_1-B_2)(1+2\lambda)^2)\vert}\\
 \quad\quad\quad\quad\quad\quad\quad\quad\quad\quad\quad\quad \quad \textsl{if}\quad (1+2\lambda)^2 B_1\leq \vert (1+4\lambda)B_1^2 +(B_1-B_2)(1+2\lambda)^2\vert,\\
 \dfrac{1}{2(1+3\lambda)(1+2\lambda)}\big(\vert B_1^2-(1+2\lambda)(B_1-B_2)\vert+ (1+2\lambda)B_1\big)\quad \textsl{if} \quad (1+2\lambda)^2 B_1\geq\\\quad\quad\quad\quad\quad\quad\quad\quad\quad\quad\quad \quad\quad\quad\quad\quad\quad\quad\quad\quad\quad\quad\quad\vert (1+4\lambda)B_1^2+(B_1-B_2)(1+2\lambda)^2\vert.
 \end{cases}
\end{equation}
 From equations (\ref{25}) and (\ref{26}), we have the desired estimate for $a_3$.\\
Now,  adding  equations (\ref{ll}) and  (\ref{68}), we have
\begin{equation}\label{rr}
(9+44\lambda)(a_2a_3-a_2^3)=\dfrac{(B_2-B_1)}{2}p_1p_2+\dfrac{(B_2-B_1)}{2}q_1q_2+\dfrac{B_1}{2}(p_3+q_3).
\end{equation}
Using (\ref{j}) and  (\ref{qq}) in (\ref{rr}), we have
 \begin{equation}\label{diff}
p_1(p_2-q_2)=\frac{8(1+3\lambda)(1+2\lambda)B_1(p_3+q_3)}{(9+44\lambda)B_1^2-8(1+3\lambda)(1+2\lambda)(B_2-B_1)}.
\end{equation}
Now, to find an estimate  for $\vert a_4 \vert$, we subtract equation (\ref{68}) from (\ref{ll}) and then using (\ref{rr}), (\ref{j}), (\ref{qq}), we get
\begin{align}
6(1+4\lambda)a_4 &=(15+60\lambda)a_2a_3-(11+48\lambda)a_2^3+\frac{(B_1+B_3-2B_2)}{4}p_1^3+\frac{B_2-B_1}{2}p_1p_2\notag\\&\quad+\frac{B_1-B_2}{2}q_1q_2+\frac{B_1}{2}(p_3-q_3)\notag\\
&=(9+44\lambda)(a_2a_3-a_2^2)+(6+16\lambda)a_2a_3-2(1+2\lambda)a_2^3+\frac{(B_1+B_3-2B_2)}{4}p_1^3\notag\\&\quad+\frac{B_2-B_1}{2}p_1p_2+\frac{B_1-B_2}{2}q_1q_2+\frac{B_1}{2}(p_3+q_3)\notag\\
 &=\frac{(B_2-B_1)}{2}p_1p_2+\frac{(B_2-B_1)}{2}q_1q_2+\frac{B_1}{2}(p_3-q_3)+(6+16\lambda)\frac{B_1}{2(1+2\lambda)}p_1\notag\\&\quad\bigg(\frac{B_1^2}{4(1+2\lambda)^2}p_1^2+\frac{B_1}{8(1+3\lambda)}(p_2-q_2)\bigg)-2(1+2\lambda)\frac{B_1^3}{8(1+2\lambda)^3}p_1^3\notag\\&\quad+\frac{(B_1+B_3-2B_2)}{4}p_1^3+\frac{B_2-B_1}{2}p_1p_2+\frac{B_1-B_2}{2}q_1q_2+\frac{B_1}{2}(p_3-q_3)\label{70}\\
&=(B_2-B_1)p_1p_2+B_1p_3+\frac{(3+8\lambda)B_1^2}{8(1+2\lambda)(1+3\lambda)}p_1p_2-\frac{(3+8\lambda)B_1^2}{8(1+2\lambda)(1+3\lambda)}p_1q_2\notag\\&\quad+\frac{1}{4}p_1^3\bigg(\frac{2(1+3\lambda)B_1^3}{(1+2\lambda)^3}+(B_1+B_3-2B_2)\bigg)\notag\\
&=(B_2-B_1)p_1p_2+B_1p_3+\frac{(3+8\lambda)B_1^2}{8(1+2\lambda)(1+3\lambda)}p_1p_2-\frac{(3+8\lambda)B_1^2}{8(1+2\lambda)(1+3\lambda)}p_1q_2\notag\\&\quad+\frac{B_1(1+2\lambda)^2}{4((1+4\lambda)B_1^2+(B_1-B_2)(1+2\lambda)^2)}\bigg(\frac{2(1+3\lambda)B_1^3}{(1+2\lambda)^3}+(B_1+B_3-2B_2)\bigg)p_1(p_2+q_2)\notag\\
&=B_1p_3+p_1p_2\bigg((B_2-B_1)+\frac{(3+8\lambda)B_1^2}{8(1+2\lambda)(1+3\lambda)}+\frac{B_1(1+2\lambda)^2}{4((1+4\lambda)B_1^2+(B_1-B_2)(1+2\lambda)^2)}\notag\\&\quad\bigg(\frac{2(1+3\lambda)B_1^3}{(1+2\lambda)^3}+(B_1+B_3-2B_2)\bigg)\bigg)+p_1q_2\bigg(-\frac{(3+8\lambda)B_1^2}{8(1+2\lambda)(1+3\lambda)}\notag\\&\quad+\frac{B_1(1+2\lambda)^2}{4((1+4\lambda)B_1^2+(B_1-B_2)(1+2\lambda)^2)}\bigg(\frac{2(1+3\lambda)B_1^3}{(1+2\lambda)^3}+(B_1+B_3-2B_2)\bigg)\bigg).\label{69}
\end{align}
Let
\begin{equation}\label{tt}
\begin{split}
 A=&(B_2-B_1)+\frac{(3+8\lambda)B_1^2}{8(1+2\lambda)(1+3\lambda)}+\frac{B_1(1+2\lambda)^2}{4((1+4\lambda)B_1^2+(B_1-B_2)(1+2\lambda)^2)}\\&\quad\bigg(\frac{2(1+3\lambda)B_1^3}{(1+2\lambda)^3}+(B_1+B_3-2B_2)\bigg),
\end{split}
\end{equation}
\begin{equation}\label{uu}
C=-\frac{(3+8\lambda)B_1^2}{8(1+2\lambda)(1+3\lambda)}+\frac{B_1(1+2\lambda)^2}{4((1+4\lambda)B_1^2+(B_1-B_2)(1+2\lambda)^2)}\bigg(\frac{2(1+3\lambda)B_1^3}{(1+2\lambda)^3}+(B_1+B_3-2B_2)\bigg).
\end{equation}
Then (\ref{69}) reduces to  \[6(1+4\lambda)a_4=B_1p_3+Ap_1p_2+Cp_1q_2.\]
Now using $\vert p_2\vert\leq 2$ and $\vert q_2\vert \leq 2$ and refined estimate for $\vert p_1\vert$ from (\ref{ppp}), we  have \\
\begin{equation}\label{xx}
\vert a_4\vert \leq \frac{1}{3(1+4\lambda)}
\begin{cases}
B_1+\dfrac {2\sqrt{B_1}(1+2\lambda)}{\sqrt{\vert(1+4\lambda)B_1^2+(B_1-B_2)(1+2\lambda)^2\vert}}(\vert A\vert+\vert C\vert) \quad \textsl{if} (1+2\lambda)^2B_1\leq \\ \quad \quad \quad \quad \quad\quad\quad\quad\quad\quad\quad\quad\quad\quad \quad\quad\vert  (1+4\lambda)B_1^2+(B_1-B_2)(1+2\lambda)^2\vert,\\
B_1+2(\vert A\vert +\vert C\vert)\quad \textsl{if}\quad  (1+2\lambda)^2B_1\geq \vert  (1+4\lambda)B_1^2+(B_1-B_2)(1+2\lambda)^2\vert.
\end{cases}
\end{equation}
where $A$ and $C$ are given by (\ref{tt}) and (\ref{uu}) respectively.

We now find another estimate for $\vert a_4\vert $. From equation (\ref{70}), we get
\begin{equation*}
\begin{split}
6(1+4\lambda)a_4 &= B_1p_3+\frac{(B_2-B_1)}{2}p_1(p_2+q_2)+\bigg(\frac{(3+8\lambda)B_1^2}{8(1+2\lambda)(1+3\lambda)}+\frac{(B_2-B_1)}{2}\bigg)\\&\quad p_1(p_2-q_2)+\frac{(2(1+3\lambda)B_1^3+(1+2\lambda)^3(B_1+B_3-2B_2))}{4(1+2\lambda)^3}p_1^3.
\end{split}
\end{equation*}
Using equations (\ref{pp}) and (\ref{diff}), we have
\begin{equation*}
\begin{split}
6(1+4\lambda)a_4 & =\frac{B_1}{(9+44\lambda))B_1^2-8(1+2\lambda)(1+3\lambda)(B_2-B_1)}((12+52\lambda)B_1^2\\&\quad -4(1+2\lambda)(1+3\lambda)(B_2-B_1))p_3+((3+8\lambda)B_1^2+4(1+2\lambda)(1+3\lambda)(B_2-B_1))q_3)\\&\quad +\frac{1}{2}\bigg((B_2-B_1)+\frac{B_1(2(1+3\lambda)B_1^3+(1+2\lambda)^3(B_1+B_3-2B_2))}{2(1+2\lambda)((1+4\lambda)B_1^2+(B_1-B_2)(1+2\lambda)^2)}\bigg)p_1(p_2+q_2).
\end{split}
\end{equation*}
Using the estimate $\vert p_i\vert\leq 2$ and $\vert q_i\vert\leq 2$ for $i=2,3$ and the refined estimate for $\vert p_1\vert $ given by (\ref{ppp}),
whenever $(1+2\lambda)^2B_1\leq \vert(1+4\lambda)B_1^2+(B_1-B_2)(1+2\lambda)^2\vert $, we get
\begin{align}
\vert a_4\vert \leq &
\dfrac{2 B_1}{ 6(1+4\lambda)\vert (9+44\lambda)B_1^2- 8(1+2\lambda)(1+3\lambda)(B_2 -B_1)\vert}\bigg(\vert (12+52\lambda)B_1^2\notag\\&\quad-4(1+2\lambda)(1+3\lambda)(B_2-B_1)\vert +
\vert(3+8\lambda)B_1^2 +4(1+2\lambda)(1+3\lambda)(B_2-B_1)\vert\bigg)\notag\\&\quad +\dfrac{2(1+2\lambda)\sqrt{B_1}}{3(1+4\lambda)\sqrt{\vert(1+4\lambda)B_1^2+(B_1-B_2)(1+2\lambda)^2\vert}}\notag\\&\bigg\vert (B_2-B_1)+\dfrac{B_1(2(1+3\lambda)B_1^3+(1+2\lambda)^3(B_1+B_3-2B_2))}{2(1+2\lambda)((1+4\lambda)B_1^2+(B_1-B_2)(1+2\lambda)^2)}\bigg\vert.\label{21}
\end{align}
And whenever $(1+2\lambda)^2B_1\geq \vert(1+4\lambda)B_1^2+(B_1-B_2)(1+2\lambda)^2\vert $, we get\\
\begin{align}
\vert a_4\vert  \leq & \dfrac{2 B_1}{6(1+44\lambda)\vert (9+44\lambda)B_1^2-8(1+2\lambda)(1+3\lambda)(B_2-B_1)\vert}\bigg(\vert (12+52\lambda)B_1^2\notag\\&\quad-4(1+2\lambda)(1+3\lambda)(B_2-B_1)\vert +\vert(3+8\lambda)B_1^2+4(1+2\lambda)(1+3\lambda)(B_2-B_1)\vert\bigg)\notag\\&\quad +\dfrac{2}{3(1+4\lambda)}\bigg\vert (B_2-B_1) +\dfrac{B_1(2(1+3\lambda)B_1^3+(1+2\lambda)^3(B_1+B_3-2B_2))}{2(1+2\lambda)((1+4\lambda)B_1^2+(B_1-B_2)(1+2\lambda)^2)}\bigg\vert.\label{22}
\end{align}
 From (\ref{xx}), (\ref{21}) and (\ref{22}), the desired estimate on $\vert a_4\vert$ follows.
\end{proof}
For $\lambda=0$, Theorem \ref{vv} readily yields the following coefficient estimates for  Ma-Minda bi-starlike functions.
\begin{cor}\label{yy}
Let $f$ given by (\ref{aa}) be in the class $ST_{\sigma}(\varphi)$.
\begin{itemize}
\item[(a)] For $B_1\leq \vert  B_1^2+B_1-B_2\vert$
\begin{align}
\vert a_2\vert \leq & \dfrac{B_1\sqrt{B_1}}{\sqrt{\vert B_1^2+B_1-B_2\vert}}\label{27},\\
\vert a_3\vert \leq &\min \dfrac{B_1}{ 2\vert B_1^2+(B_1-B_2)\vert} \bigg\lbrace\dfrac{1}{2} \bigg(\vert (3B_1^2+(B_1-B_2)\vert +\vert B_1^2-(B_1-B_2)\vert\bigg),\notag\\&\quad \vert B_1^2-(B_1-B_2)\vert+\vert B_1^2+(B_1-B_2)\vert\bigg\rbrace,\label{28}
 &\intertext{and}
\vert a_4\vert \leq &
\min \bigg\lbrace\dfrac{B_1}{3}+\dfrac {2\sqrt{B_1}}{3\sqrt{\vert B_1^2+B_1-B_2\vert}}(\vert A\vert+\vert C
\vert),\notag\\&\quad\quad\dfrac{ B_1}{3\vert 9B_1^2-8(B_2-B_1)\vert}\bigg(\vert 12 B_1^2-4(B_2-B_1)\vert +\vert 3B_1^2+4(B_2- B_1)\vert\bigg)\notag\\&\quad\quad +\dfrac{2\sqrt{B_1}}{3\sqrt{\vert B_1^2+(B_1-B_2)\vert}}\bigg\vert (B_2-B_1)+\dfrac{B_1(2 B_1^3+(B_1+B_3-2B_2))}{2(B_1^2+(B_1-B_2))}\bigg\vert\bigg\rbrace.\label{z}
\end{align}
\item[(b)] For  $B_1\geq \vert  B_1^2+B_1-B_2\vert$
\begin{align*}
\vert a_2\vert \leq & B_1,\\
\vert a_3\vert \leq & \min \bigg\lbrace \dfrac{B_1}{ 4\vert B_1^2+(B_1-B_2)\vert}\bigg(\vert 3 B_1^2+(B_1-B_2)\vert+\vert B_1^2-(B_1-B_2)\vert\bigg), \\&\quad\quad\dfrac{1}{2}\big(\vert B_1^2-(B_1-B_2)\vert+ B_1\big)\bigg\rbrace,\\
\vert a_4\vert \leq &
\min \bigg\lbrace\dfrac{B_1}{3}+\dfrac {2}{3}(\vert A\vert+\vert C), \\&\quad\quad\dfrac{ B_1}{3\vert 9B_1^2-8(B_2-B_1)\vert}\bigg(\vert 12B_1^2-4(B_2-B_1)\vert +\vert 3
 B_1^2+4(B_2- B_1)\vert\bigg)\\&\quad\quad +\frac{2}{3}\bigg\vert (B_2-B_1)+\dfrac{B_1(2 B_1^3+(B_1+B_3-2B_2))}{2(B_1^2+(B_1-B_2))}\bigg\vert\bigg\rbrace.
 \end{align*}
 \end{itemize}
where \[A=\frac{7 {B_1}^4-5 {B_1}^3+5 {B_1}^2 {B_2}-6 {B_1}^2+12 {B_1} {B_2}+2 {B_1}{B_3}-8 {B_2}^2}{8 \left({B_1}^2+{B_1}-{B_2}\right)}\] and \[C=
 \frac{{B_1} \left({B_1}^3-3 {B_1}^2+3 {B_1} {B_2}+2 {B_1}-4 {B_2}+2 {B_3}\right)}{8 \left({B_1}^2+{B_1}-{B_2}\right)}.\]
\end{cor}

\begin{rmk}\label{46}
 Corollary  \ref{yy}  improves the bounds obtained for $\vert a_2\vert$ and $\vert a_3\vert$ in \cite[ Corollary 2.1]{ali}.
\end{rmk}
\begin{rmk}\label{74}
 For the class of bi-starlike functions of order $\rho$, $0\leq \rho<1$, the function $\varphi$ is given by  \[\varphi(z)=\frac{1+(1-2\rho)z}{1-z} = 1+ 2(1-\rho )z+2(1-\rho)z^2+2(1-\rho) z^3+\cdots \] and so $B_1=B_2=B_3=2(1-\rho)$.
Hence from Corollary \ref{yy}, we have the following estimates for $\vert a_2\vert $ and $\vert a_3\vert$ which coincides with the result (for $\lambda=0$) \cite[Theorem 1]{as} of Mishra and Barik.
\begin{equation*}
\vert a_2\vert \leq
\begin{cases}
\sqrt{2(1-\rho)}  \quad \textsl{if}\quad 0\leq\rho\leq1/2,\\
2(1-\rho) \quad \textsl{if}\quad 1/2\leq \rho<1.
\end{cases}
\end{equation*}
and
\begin{equation*}
\vert a_3\vert \leq
\begin{cases}
2(1-\rho) \quad \textsl{if}\quad 0\leq \rho\leq 1/2,\\
(1-\rho)(3-2\rho)  \quad \textsl{if}\quad 1/2\leq\rho <1.
\end{cases}
\end{equation*}
Further, the inequality (\ref{z}) for  $B_1=B_2=B_3=2(1-\rho)$ reduces to
\begin{equation*}
\vert a_4\vert\leq
\begin{cases}
\min\bigg\lbrace \dfrac{2(1-\rho)}{3}(1+2\sqrt{2(1-\rho)}), \dfrac{2(1-\rho)}{3}(\frac{5}{3}+ 2 \sqrt{2(1-\rho ))}\bigg\rbrace\quad \textsl{if} \quad 0\leq \rho \leq 1/2,\\
\min\bigg\lbrace\dfrac{2(1-\rho)}{3}(1+4(1-\rho),\dfrac{2(1-\rho)}{3}(\dfrac{5}{3}+4(1-\rho))\bigg\rbrace\quad \textsl{if} \quad 1/2\leq \rho <1.
\end{cases}
\end{equation*}
which gives
\begin{equation*}
\vert a_4\vert\leq
\begin{cases}
 \dfrac{2(1-\rho)}{3}(1+2\sqrt{2(1-\rho)}\quad \textsl{if} \quad 0\leq \beta \leq 1/2,\\
\dfrac{2(1-\rho)}{3}(1+4(1-\rho)\quad \textsl{if} \quad 1/2\leq \beta <1.
\end{cases}
\end{equation*}
This is the estimate given by Mishra and Soren  in \cite [Theorem 2.2]{ak}.\\
\end{rmk}
\begin{rmk}
 For the class of strongly bi-starlike functions of order $\beta$, $0< \beta\leq 1$, the function $\varphi$ is given by \[\varphi(z)=\bigg(\frac{1+z}{1-z}\bigg)^\beta = 1+ 2\beta z+2\beta^2z^2+\frac{2}{3} \left(2 \beta ^3+\beta \right) z^3+\cdots \] and so $B_1=2\beta$, $B_2=2\beta^2$, $B_3=\frac{4}{3}\beta^{3}+\frac{2}{3}\beta$. Also,
for strongly bi-starlike functions   $B_1\ngeq \vert  B_1^2+B_1-B_2\vert$. Equation (\ref{27})
of  Corollary \ref{yy},  gives
 \begin{equation*}
 \vert a_2\vert \leq \frac{2\beta}{\sqrt{1+\beta}}
  \end{equation*}
  which coincides with the estimate given by Brannan and Taha in \cite{bt}.
Also,  (\ref{28}) reduces to the following estimate given by Mishra and Soren  in \cite[ Theorem 2.1]{ak}:
\begin{equation*}
\vert a_3\vert \leq
\begin{cases}
\beta \quad \textsl{if}\quad 0<\beta\leq 1/3,\\
\dfrac{4\beta^2}{1+\beta}  \quad \textsl{if}\quad 1/3\leq\beta\leq 1.
\end{cases}
\end{equation*}
Equation (\ref{z}) gives
\begin{equation*}
\vert a_4\vert \leq
\begin{cases}
\dfrac{2\beta}{3}  \left(1-\dfrac{2 \left(16 \beta ^2-3 \beta -1\right)}{3 (\beta +1)^{3/2}}\right)\quad \textsl{if} \quad  0<\beta<\dfrac{3+\sqrt{37}}{32},\\
\dfrac{2\beta}{3}  \left(1+\dfrac{2 \left(16 \beta ^2-3 \beta -1\right)}{3 (\beta +1)^{3/2}}\right)\quad \textsl{if} \quad  \dfrac{3+\sqrt{37}}{32}\leq\beta\leq 1.
\end{cases}
\end{equation*}
And, this improves the estimate for the fourth coefficient given by Mishra and Soren in \cite [ Theorem 2.1]{ak}.
\end{rmk}

\begin{Def}
For $\lambda\geq 0$, the class $M ^{\lambda}(\varphi)$  consists of functions $f\in \mathcal{A}$  satisfying
\[ \quad \lambda\bigg(1+\frac{zf''(z)}{ f'(z)}\bigg)+(1-\lambda)\frac {z f'(z)}{f'(z)}\prec \varphi \quad (z\in\mathbb{D}).\] The class $M_{\sigma} ^{\lambda}(\varphi)$ consists of $f\in \sigma$ such that $f,g\in M ^{\lambda}(\varphi)$
 where $g$ is the analytic continuation of $f^{-1}$ to the unit disk $\mathbb{D}$.
\end{Def}

\begin{thm}
Let $f \in M_{\sigma} ^{\lambda}(\phi)$. \\
\begin{itemize}
\item[(a)] If $(1+\lambda)B_1\leq \vert B_1^2+(1+\lambda)(B_1-B_2)\vert$, then\\
the  coefficients $a_2$, $a_3$ and $a_4$ satisfies
\begin{align*}
\vert a_2\vert \leq &
\dfrac{B_1\sqrt{B_1}}{\sqrt{(1+\lambda)\vert B_1^2+(1+\lambda)(B_1-B_2)\vert}},\\
 \vert a_3\vert \leq & \min \dfrac{B_1}{2(1+2\lambda)(1+\lambda)\vert B_1^2+(B_1-B_2)(1+\lambda)\vert}\bigg\{ \frac{1}{2}\big(\vert(3+5\lambda)B_1^2+(1+\lambda)^2(B_1-B_2)\vert\\&\quad  +\vert(1+3\lambda)B_1^2-(1+\lambda)^2(B_1-B_2)\vert\big),\big(\vert(1+3\lambda)B_1^2-(1+\lambda)^2(B_1-B_2)\vert\\&\quad+\vert(1+\lambda)B_1^2+(1+\lambda)^2(B_1-B_2)\vert\big)\bigg\},\\
\vert a_4\vert \leq & \min\dfrac{1}{3(1+3\lambda)} \bigg\{B_1+  \dfrac{2\sqrt{(1+\lambda)B_1}}{\sqrt{\vert B_1^2+(1+\lambda)(B_1-B_2)\vert}}( \vert A_1\vert+\vert C_1\vert),\\&\quad
B_1\bigg(\bigg\vert\frac{(12+30\lambda)B_1^2-4(1+\lambda)(1+2\lambda)(B_2-B_1)}{(9+15\lambda)B_1^2-8(1+2\lambda)(1+\lambda)(B_2-B_1)}\bigg\vert\\&\quad+\bigg\vert\frac{4(1+\lambda)(1+2\lambda)(B_2-B_1)+3{B_1}^2 (1+5 \lambda )}{(9+15\lambda)B_1^2-8(1+2\lambda)(1+\lambda)(B_2-B_1)}\bigg\vert\bigg) + \dfrac{2\sqrt{(1+\lambda)B_1}}{\sqrt{\vert B_1^2+(1+\lambda)(B_1-B_2)\vert}}\\&\quad\bigg(\bigg\vert B_2-B_1+ \frac{B_1((B_1+B_3-2B_2)(1+\lambda)^3+2(1+4\lambda)B_1^3)}{2(1+\lambda)^2(B_1^2+(1+\lambda)(B_1-B_2))}\bigg\vert\bigg)\bigg\}.
\end{align*}
\item[(b)] If $(1+\lambda)B_1\leq \vert B_1^2+(1+\lambda)(B_1-B_2)\vert$, then\\
the  coefficients $a_2$, $a_3$ and $a_4$ satisfies
\begin{align*}
\vert a_2\vert \leq & \dfrac{B_1}{1+\lambda},\\
\vert a_3\vert \leq &\min  \bigg\{\frac{B_1}{4(1+2\lambda)(1+\lambda)\vert B_1^2+(B_1-B_2)(1+\lambda)\vert}\\&\quad\quad(\vert(3+5\lambda)B_1^2+(1+\lambda)^2(B_1-B_2)\vert  +\vert(1+3\lambda)B_1^2-(1+\lambda)^2(B_1-B_2)\vert),\\&\quad\quad \dfrac{1}{2(1+2\lambda)(1+\lambda)^2}(\vert (1+3\lambda)B_1^2-(1+\lambda)^2(B_1-B_2)\vert+B_1(1+\lambda)^2)\bigg\},\\
\vert a_4\vert \leq &\min \dfrac{1}{3(1+3\lambda)}\bigg\{\bigg(B_1+2(\vert A_1\vert +\vert C_1\vert )\bigg),\\&\quad\quad
B_1\bigg(\bigg\vert\frac{(12+30\lambda)B_1^2-4(1+\lambda)(1+2\lambda)(B_2-B_1)}{(9+15\lambda)B_1^2-8(1+2\lambda)(1+\lambda)(B_2-B_1)}\bigg\vert\\&\quad\quad+\bigg\vert\frac{4(1+\lambda)(1+2\lambda)(B_2-B_1)+3{B_1}^2 (1+5 \lambda )}{(9+15\lambda)B_1^2-8(1+2\lambda)(1+\lambda)(B_2-B_1)}\bigg\vert\bigg)\\&\quad\quad +2\bigg(\bigg\vert B_2-B_1+ \frac{B_1((B_1+B_3-2B_2)(1+\lambda)^3+2(1+4\lambda)B_1^3)}{2(1+\lambda)^2(B_1^2+(1+\lambda)(B_1-B_2))}\bigg\vert\bigg)\bigg\}.
\end{align*}
\end{itemize}
where the constants $A_1$ and $C_1$ are given by
\begin{align*}
 A_1=&(B_2-B_1)+\dfrac{3(1+5\lambda)B_1^2}{8(1+\lambda)(1+2\lambda)}+\dfrac{B_1(1+\lambda)}{4(B_1^2+(1+\lambda)(B_1-B_2))}\\&\quad\bigg((B_1+B_3-2B_2)+\dfrac{2(1+4\lambda)B_1^3}{(1+\lambda)^3}\bigg),\\
& \text{and}\\
C_1=&-\dfrac{3(1+5\lambda)B_1^2}{8(1+\lambda)(1+2\lambda)}+\dfrac{B_1(1+\lambda)}{4(B_1^2+(1+\lambda)(B_1-B_2))}\\&\quad\bigg((B_1+B_3-2B_2)+\dfrac{2(1+4\lambda)B_1^3}{(1+\lambda)^3}\bigg).
\end{align*}
\end{thm}
\begin{proof}
Since $f\in M_{\sigma} ^{\lambda}(\phi)$, there exists two analytic functions $r,s:\mathbb{D}\rightarrow \mathbb{D}$, with $r(0)=0=s(0)$, such that
\begin{equation}\label{30}
 \lambda\bigg(1+\frac{zf''(z)}{ f'(z)}\bigg)+(1-\lambda)\frac {z f'(z)}{f(z)}= \varphi(r(z))\quad \textsl{and} \quad \lambda\bigg(1+\frac{wg''(w)}{ g'(w)}\bigg)+(1-\lambda)\frac {w g'(w)}{g(w)}= \varphi(s(w))
\end{equation}
Define the functions $p$ and $q$ by
\begin{equation*}
p(z)=\frac{1+r(z)}{1-r(z)}=1+p_1z+p_2z^2+p_3z^3+p_4z^4+\cdots
\end{equation*}
and
\begin{equation*}
q(w)=\frac{1+s(w)}{1-s(w)}=1+q_1w+q_2w^2+q_3w^3+q_4w^4+\cdots
\end{equation*}
Proceeding as in the previous theorem, we have the following set of equations
\begin{align}
(1+\lambda)a_2&=\frac{1}{2}B_1p_1\label{38}\\
2(1+2\lambda)a_3-(1+3\lambda)a_2^2&=\frac{1}{2}B_1\bigg(p_2-\frac{1}{2} p_1^2\bigg)+\frac{1}{4}B_2p_1^2\label{39}\\
3(1+3\lambda)a_4-(3+15\lambda)a_2a_3+(1+7\lambda)a_2^3&=\frac{1}{2}B_1\bigg(\frac{p_1^3}{4}-p_1p_2+p_3\bigg)\notag+\frac{1}{2}B_2p_1\bigg(p_2-\frac{p_1^2}{2}\bigg)+\\&\quad\quad\frac{1}{8}B_3p_1^3\label{40}\\
-(1+\lambda)a_2&=\frac{1}{2}B_1q_1\label{41}\\
-2(1+2\lambda)a_3+(3+5\lambda)a_2^2&=\frac{1}{2}B_1\bigg(q_2-\frac{1}{2} q_1^2\bigg)+\frac{1}{4}B_2q_1^2\label{42}\\
-3(1+3\lambda)a_4+(12+30\lambda)a_2a_3-(10+22\lambda)a_2^3 & =\frac{1}{2}B_1\bigg(\frac{q_1^3}{4}-q_1q_2+q_3\bigg)\notag+\frac{1}{2} B_2q_1\bigg(q_2-\frac{q_1^2}{2}\bigg)\\&\quad\quad+\frac{1}{8}B_3q_1^3.\label{43}
\end{align}
Equations (\ref{38}) and (\ref{41}) gives $p_1=-q_1$. Adding equations (\ref{39}) and (\ref{42}), and then substituting $a_2= B_1p_1/2(1+\lambda)$, we  get
\begin{equation}\label{44}
p_1^2= \frac{B_1(p_2+q_2)(1+\lambda)}{B_1^2+(1+\lambda)(B_1-B_2)}.
\end{equation}
Using the well known estimates $\vert p_2\vert \leq 2$ and $\vert q_2\vert \leq 2$, we get
\begin{equation}\label{73}
\vert p_1\vert\leq
\begin{cases}
 \dfrac{2\sqrt{(1+\lambda)B_1}}{\sqrt{\vert B_1^2+(1+\lambda)(B_1-B_2)\vert}} \quad \textsl{if} \quad (1+\lambda)B_1\leq \vert B_1^2+(1+\lambda)(B_1-B_2)\vert,\\
 2\quad\quad\quad\quad\quad\quad\quad\quad\quad\quad\quad\quad \textsl{if}\quad   (1+\lambda)B_1\geq \vert B_1^2+(1+\lambda)(B_1-B_2)\vert.
  \end{cases}
\end{equation}
Using equation (\ref{38}) and the above refined estimate for $\vert p_1\vert $, we have the following estimate for $\vert a_2\vert$:
\begin{equation*}
\vert a_2\vert \leq
\begin{cases}
\dfrac{B_1\sqrt{B_1}}{\sqrt{(1+\lambda)\vert B_1^2+(1+\lambda)(B_1-B_2)\vert}} \quad \textsl{if} \quad (1+\lambda)B_1\leq \vert B_1^2+(1+\lambda)(B_1-B_2)\vert,\\
\dfrac{B_1}{1+\lambda}\quad\quad\quad\quad\quad\quad\quad\quad\quad\quad \textsl{if}\quad   (1+\lambda)B_1\geq \vert B_1^2+(1+\lambda)(B_1-B_2)\vert.
\end{cases}
\end{equation*}
Also, equation (\ref{44}) gives \[p_2+q_2=\frac{(B_1^2+(1+\lambda)(B_1-B_2))p_1^2}{B_1(1+\lambda)}.\]
Thus, we have
\begin{equation}\label{51}
\vert p_2+q_2\vert \leq
\begin{cases}
4\quad\quad\quad\quad\quad\quad\quad\quad\quad\quad\quad\quad \textsl{if}\quad   (1+\lambda)B_1\leq \vert B_1^2+(1+\lambda)(B_1-B_2)\vert,\\
\dfrac{4\vert B_1^2+(1+\lambda)(B_1-B_2)\vert}{B_1(1+\lambda)} \quad \textsl{if} \quad (1+\lambda)B_1\geq \vert B_1^2+(1+\lambda)(B_1-B_2)\vert.
\end{cases}
\end{equation}
Now, to find an estimate for $\vert a_3\vert$, we subtract (\ref{42}) from (\ref{39}) and get
\begin{equation}\label{48}
a_3=a_2^2+\frac{B_1}{8(1+2\lambda)}(p_2-q_2).
\end{equation}
Using (\ref{38}) and then (\ref{44}) in (\ref{48}), we get
\begin{align}
a_3=& \frac{B_1}{8(1+2\lambda)(1+\lambda)(B_1^2+(B_1-B_2)(1+\lambda))}(((3+5\lambda)B_1^2+(1+\lambda)^2(B_1-B_2))p_2\nonumber\\&+((1+3\lambda)B_1^2-(1+\lambda)^2(B_1-B_2))q_2)\label{49}\\
=& \frac{B_1}{8(1+2\lambda)(1+\lambda)(B_1^2+(B_1-B_2)(1+\lambda))}(((1+3\lambda)B_1^2-(1+\lambda)^2(B_1-B_2))(p_2+q_2)\nonumber\\&+2((1+\lambda)B_1^2+(1+\lambda)^2(B_1-B_2))p_2).\label{50}
\end{align}
Using $\vert p_2\vert\leq 2$ and $\vert q_2\vert \leq 2$ in (\ref{49}), we have the first estimate for $\vert a_3\vert$ given by
\begin{equation}\label{52}
\begin{split}
\vert a_3\vert \leq &
\frac{B_1}{4(1+2\lambda)(1+\lambda)\vert B_1^2+(B_1-B_2)(1+\lambda)\vert}(\vert(3+5\lambda)B_1^2+(1+
\lambda)^2(B_1-B_2)\vert \\& +\vert(1+3\lambda)B_1^2-(1+\lambda)^2(B_1-B_2)\vert).
\end{split}
\end{equation}
Also using the estimate for $\vert p_2+q_2\vert $ from (\ref{51}) in (\ref{50}), we get another estimate for $\vert a_3\vert$ given by
\begin{equation}\label{53}
\vert a_3\vert \leq
\begin{cases}
\dfrac{B_1}{2(1+2\lambda)(1+\lambda)\vert B_1^2+(B_1-B_2)(1+\lambda)\vert}(\vert(1+3\lambda)B_1^2-(1+\lambda)^2(B_1-B_2)\vert\\\quad\quad\quad+\vert(1+\lambda)B_1^2+(1+\lambda)^2(B_1-B_2)\vert) \quad \textsl{if} \quad (1+\lambda)B_1\leq \vert B_1^2+(1+\lambda)(B_1-B_2)\vert,\\
\dfrac{1}{2(1+2\lambda)((1+\lambda)^2}(\vert (1+3\lambda)B_1^2-(1+\lambda)^2(B_1-B_2)\vert+B_1(1+\lambda)^2) \\\quad\quad\quad\quad\quad\quad\quad\quad\quad\quad\quad\quad\quad\quad\quad\quad \textsl{if} \quad (1+\lambda)B_1\geq \vert B_1^2+(1+\lambda)(B_1-B_2)\vert.
\end{cases}
\end{equation}
Using (\ref{52}) and (\ref{53}), desired estimate for $a_3$ follows.\\

\noindent Now adding equations (\ref{40}) and (\ref{43}), we get
\begin{equation}\label{54}
(9+15\lambda)(a_2a_3-a_2^3)= \frac{(B_2-B_1)}{2}p_1p_2+\frac{(B_2-B_1)}{2}q_1q_2+\frac{B_1}{2}(p_3+q_3).
\end{equation}
Substituting the value of $a_3$ from (\ref{48}) and using $p_1=-q_1$ in (\ref{54}), we get
\begin{equation}\label{58}
p_1(p_2-q_2)=\frac{8(1+2\lambda)(1+\lambda)B_1(p_3+q_3)}{(9+15\lambda)B_1^2-8(1+2\lambda)(1+\lambda)(B_2-B_1)}
\end{equation}
To find an estimate for $\vert a_4\vert$,  subtracting (\ref{43}) from (\ref{40}), we get
\begin{align}\label{72}
6(1+3\lambda)a_4 &=(15+45\lambda)a_2a_3-(11+29\lambda)a_2^3+\frac{(B_1+B_3-2B_2)}{4}p_1^3+\frac{B_2-B_1}{2}p_1p_2\notag\\&\quad+\frac{B_1-B_2}{2}q_1q_2+\frac{B_1}{2}(p_3-q_3)\notag\\
&=(9+15\lambda)(a_2a_3-a_2^2)+(6+30\lambda)a_2a_3-2(1+7\lambda)a_2^3+\frac{(B_1+B_3-2B_2)}{4}p_1^3\notag\\&\quad+\frac{B_2-B_1}{2}p_1p_2+\frac{B_1-B_2}{2}q_1q_2+\frac{B_1}{2}(p_3+q_3).
\end{align}
 We replace $(9+15\lambda)(a_2a_3-a_2^2)$ by the right hand side of (\ref{54}) and use the values of $a_2$ and $a_3$ from equations (\ref{38}) and (\ref{48}) in (\ref{72}), we get
\begin{equation*}
\begin{split}
6(1+3\lambda)a_4 &=\frac{(B_2-B_1)}{2}p_1p_2+\frac{(B_2-B_1)}{2}q_1q_2+\frac{B_1}{2}
(p_3+q_3)+(6+30\lambda)\frac{B_1p_1}{2(1+\lambda)}\\&\quad\times\bigg(\frac{B_1^2p_1^2}{4(1+\lambda)^2}+\frac{B_1}{8(1+2\lambda)}(p_2-q_2)\bigg)-\frac{2(1+7\lambda)B_1^3}{8(1+\lambda)^3}p_1^3\\&\quad+\frac{(B_1+B_3-2B_2)}{4}
p_1^3+\frac{B_2-B_1}{2}p_1p_2+\frac{B_1-B_2}{2}q_1q_2+\frac{B_1}{2}(p_3+q_3)\\
&= B_1p_3+p_1p_2(B_2-B_1)+\frac{1}{4}\bigg((B_1+B_3-2B_2)+\frac{2(1+4\lambda)B_1^3}{(1+\lambda)^3}\bigg)p_1^3\\&\quad+\frac{3{B_1}^2 (1+5 \lambda ) p_1 (p_2-q_2)}{8 (1+\lambda ) (1+2 \lambda)}.
\end{split}
\end{equation*}
Using (\ref{44}) in above equation, we get
\begin{align}
6(1+3\lambda)a_4 =& B_1p_3+p_1p_2(B_2-B_1)+\frac{B_1(1+\lambda)}{4(B_1^2+(1+\lambda)(B_1-B_2)}\bigg((B_1+B_3-2B_2)\notag\\&\quad+\frac{2(1+4\lambda)B_1^3}{(1+\lambda)^3}\bigg)p_1(p_2+q_2)+\frac{3{B_1}^2 (1+5 \lambda )}{8 (1+\lambda ) (1+2 \lambda)} p_1 (p_2-q_2)\label{57}\\
= & B_1p_3+p_1p_2\bigg((B_2-B_1)+\frac{3(1+5\lambda)B_1^2}{8(1+\lambda)(1+2\lambda)}+\frac{B_1(1+\lambda)}{4(B_1^2+(1+\lambda)(B_1-B_2))}\notag\\&\quad\bigg((B_1+B_3-2B_2)+\frac{2(1+4\lambda)B_1^3}{(1+\lambda)^3}\bigg)\bigg)+q_1q_2\bigg(-\frac{3(1+5\lambda)B_1^2}{8(1+\lambda)(1+2\lambda)}\notag\\&\quad+\frac{B_1(1+\lambda)}{4(B_1^2+(1+\lambda)(B_1-B_2))}\bigg((B_1+B_3-2B_2)+\frac{2(1+4\lambda)B_1^3}{(1+\lambda)^3}\bigg)\bigg)\label{56}.
\end{align}
Let
\begin{align}
 A_1=&(B_2-B_1)+\dfrac{3(1+5\lambda)B_1^2}{8(1+\lambda)(1+2\lambda)}+\dfrac{B_1(1+\lambda)}{4(B_1^2+(1+\lambda)(B_1-B_2))}\\&\notag\quad\bigg((B_1+B_3-2B_2)+\dfrac{2(1+4\lambda)B_1^3}{(1+\lambda)^3}\bigg),\\
&\text{and}\notag\\
C_1=&-\dfrac{3(1+5\lambda)B_1^2}{8(1+\lambda)(1+2\lambda)}+\dfrac{B_1(1+\lambda)}{4(B_1^2+(1+\lambda)(B_1-B_2))}\\&\notag\quad\bigg((B_1+B_3-2B_2)+\dfrac{2(1+4\lambda)B_1^3}{(1+\lambda)^3}\bigg).
\end{align}
Thus (\ref{56}) reduces to
\begin{equation}\label{55}
6(1+3\lambda)a_4 = B_1p_3+A_1p_1p_2+C_1p_1q_2
\end{equation}
Using the refined estimate for $\vert p_1\vert$ given by (\ref{73}) and the well known bounds $\vert p_i\vert\leq 2$ for $i=2,3$ $\vert q_2\vert \leq 2$ in (\ref{55}), we get
\begin{equation}\label{63}
\vert a_4\vert \leq \dfrac{1}{3(1+3\lambda)}
\begin{cases}
B_1+  \dfrac{2\sqrt{(1+\lambda)B_1}}{\sqrt{\vert B_1^2+(1+\lambda)(B_1-B_2)\vert}}( \vert A_1\vert+\vert C_1\vert)\quad \textsl{if} \quad (1+\lambda)B_1\leq \\\quad\quad\quad\quad\quad\quad\quad\quad\quad\quad\quad\quad\quad\quad\quad\quad\quad\quad\vert B_1^2+(1+\lambda)(B_1-B_2)\vert,\\
B_1+2(\vert A_1\vert +\vert C_1\vert ) \quad \textsl{if} \quad\quad\quad (1+\lambda)B_1\leq \vert B_1^2+(1+\lambda)(B_1-B_2)\vert.\\
\end{cases}
\end{equation}
To find another estimate for $|a_4|$, we rewrite equation (\ref{57}) as
\begin{align*}
6(1+3\lambda)a_4 &= B_1p_3+\frac{(B_2-B_1)}{2}p_1(p_2+q_2)+\frac{(B_2-B_1)}{2}p_1(p_2-q_2)\\&\quad+\frac{B_1(1+\lambda)}{4(B_1^2+(1+\lambda)(B_1-B_2))}\bigg((B_1+B_3-2B_2)+\frac{2(1+4\lambda)B_1^3}{(1+\lambda)^3}\bigg)p_1(p_2+q_2)\\&\quad+\frac{3{B_1}^2 (1+5 \lambda )}{8 (1+\lambda ) (1+2 \lambda)} p_1 (p_2-q_2).\\
&=  B_1p_3+\bigg(\frac{B_2-B_1}{2}+ \frac{B_1(1+\lambda)}{4(B_1^2+(1+\lambda)(B_1-B_2))}\bigg((B_1+B_3-2B_2)+\frac{2(1+4\lambda)B_1^3}{(1+\lambda)^3}\bigg)\bigg)\\&\quad  p_1(p_2+q_2)+\bigg(\frac{B_2-B_1}{2}+\frac{3{B_1}^2 (1+5 \lambda )}{8 (1+\lambda ) (1+2 \lambda)}\bigg) p_1 (p_2-q_2).
\end{align*}
Replacing $p_1(p_2-q_2)$  in above equation with the right hand side of (\ref{58}), we get
 \begin{align}
6(1+3\lambda)a_4 &= B_1p_3+\bigg(\frac{B_2-B_1}{2}+ \frac{B_1((B_1+B_3-2B_2)(1+\lambda)^3+2(1+4\lambda)B_1^3)}{4(1+\lambda)^2(B_1^2+(1+\lambda)(B_1-B_2))}\bigg)  p_1(p_2+q_2)\notag\\&\quad+\frac{B_1(4(1+\lambda)(1+2\lambda)(B_2-B_1)+3{B_1}^2 (1+5 \lambda ))}{(9+15\lambda)B_1^2-8(1+2\lambda)(1+\lambda)(B_2-B_1)}(p_3+q_3)\notag\\
&= B_1p_3\bigg(\frac{(12+30\lambda)B_1^2-4(1+\lambda)(1+2\lambda)(B_2-B_1)}{(9+15\lambda)B_1^2-8(1+2\lambda)(1+\lambda)(B_2-B_1)}\bigg)\notag\\&\quad+B_1q_3\bigg(\frac{(4(1+\lambda)(1+2\lambda)(B_2-B_1)+3{B_1}^2 (1+5 \lambda ))}{(9+15\lambda)B_1^2-8(1+2\lambda)(1+\lambda)(B_2-B_1)}\notag\\&\quad +\bigg(\frac{B_2-B_1}{2}+ \frac{B_1((B_1+B_3-2B_2)(1+\lambda)^3+2(1+4\lambda)B_1^3)}{4(1+\lambda)^2(B_1^2+(1+\lambda)(B_1-B_2))}\bigg)  p_1(p_2+q_2).\label{60}
\end{align}
Using the refined estimate for $|p_1|$ given by (\ref{73}) and the inequalities $\vert p_i\vert \leq 2$, $\vert q_i\vert \leq 2$ for $i=2,3$ in (\ref{60}), we get \\
whenever $(1+\lambda)B_1\leq \vert B_1^2+(1+\lambda)(B_1-B_2)\vert$, then
\begin{equation}\label{61}
\begin{split}
\vert a_4\vert \leq &\dfrac{1}{3(1+3\lambda)}\bigg(
B_1\bigg(\bigg\vert\frac{(12+30\lambda)B_1^2-4(1+\lambda)(1+2\lambda)(B_2-B_1)}{(9+15\lambda)B_1^2-8(1+2\lambda)(1+\lambda)(B_2-B_1)}\bigg\vert\\&+\bigg\vert(\frac{(4(1+\lambda)(1+2\lambda)(B_2-B_1)+3{B_1}^2 (1+5 \lambda ))}{(9+15\lambda)B_1^2-8(1+2\lambda)(1+\lambda)(B_2-B_1)}\bigg\vert\bigg) + \dfrac{2\sqrt{(1+\lambda)B_1}}{\sqrt{\vert B_1^2+(1+\lambda)(B_1-B_2)\vert}}\\&\quad\times\bigg(\bigg\vert B_2-B_1+ \frac{B_1((B_1+B_3-2B_2)(1+\lambda)^3+2(1+4\lambda)B_1^3)}{2(1+\lambda)^2(B_1^2+(1+\lambda)(B_1-B_2))}\bigg\vert\bigg)\bigg),
\end{split}
\end{equation}
and whenever $(1+\lambda)B_1\geq \vert B_1^2+(1+\lambda)(B_1-B_2)\vert$, then
\begin{equation}\label{62}
\begin{split}
\vert a_4\vert \leq &\dfrac{1}{3(1+3\lambda)}\bigg(
B_1\bigg(\bigg\vert\frac{(12+30\lambda)B_1^2-4(1+\lambda)(1+2\lambda)(B_2-B_1)}{(9+15\lambda)B_1^2-8(1+2\lambda)(1+\lambda)(B_2-B_1)}\bigg\vert\\&+\bigg\vert(\frac{(4(1+\lambda)(1+2\lambda)(B_2-B_1)+3{B_1}^2 (1+5 \lambda ))}{(9+15\lambda)B_1^2-8(1+2\lambda)(1+\lambda)(B_2-B_1)}\bigg\vert\bigg)\\& +2\bigg(\bigg\vert B_2-B_1+ \frac{B_1((B_1+B_3-2B_2)(1+\lambda)^3+2(1+4\lambda)B_1^3)}{2(1+\lambda)^2(B_1^2+(1+\lambda)(B_1-B_2))}\bigg\vert\bigg)\bigg).
\end{split}
\end{equation}
From (\ref{63}), (\ref{61}) and (\ref{62}), desired estimate for $\vert a_4\vert $ follows.
\end{proof}
\begin{cor}
Let $f\in CV_{\sigma}(\phi)$.
\begin{itemize}
\item[(a)]If $2B_1\leq \vert B_1^2+2(B_1-B_2)\vert$, then
\begin{align*}
\vert a_2\vert \leq &
\dfrac{B_1\sqrt{B_1}}{\sqrt{2\vert B_1^2+2(B_1-B_2)\vert}},\\
\vert a_3\vert \leq & \min  \frac{B_1}{6\vert B_1^2+2(B_1-B_2)\vert}\bigg\{ \vert 2B_1^2+(B_1-B_2)\vert +\vert B_1^2-(B_1-B_2)\vert,\\&\quad\quad\vert 2B_1^2-2(B_1-B_2)\vert+\vert B_1^2+2(B_1-B_2)\vert\bigg\},\intertext{and}
\vert a_4\vert \leq &\min \bigg\{\dfrac{1}{12}\bigg(B_1+  \dfrac{2\sqrt{2B_1}}{\sqrt{\vert B_1^2+2(B_1-B_2)\vert}}( \vert A_1\vert+\vert C_1\vert)\bigg),\\&\quad\quad\dfrac{1}{24}\bigg(
\dfrac{B_1}{2}\bigg(\bigg\vert\frac{7B_1^2-4(B_2-B_1)}{B_1^2-2(B_2-B_1)}\bigg\vert+\bigg\vert\frac{3{B_1}^2+4(B_2-B_1) }{B_1^2-2(B_2-B_1)}\bigg\vert\bigg)\\&\quad\quad + \dfrac{8\sqrt{2B_1}}{\sqrt{\vert B_1^2+2(B_1-B_2)\vert}}\bigg(\bigg\vert\frac{B_2-B_1}{2}+ \frac{B_1(8(B_1+B_3-2B_2)+10 B_1^3)}{16(B_1^2+2(B_1-B_2))}\bigg\vert\bigg)\bigg)\bigg\}.
\end{align*}
 \item[(b)]If  $2B_1\geq \vert B_1^2+2(B_1-B_2)\vert$, then
\begin{align*}
\vert a_2\vert \leq& \dfrac{B_1}{2},\\
\vert a_3\vert \leq& \min  \bigg\{\frac{B_1}{6\vert B_1^2+2(B_1-B_2)\vert}(\vert 2B_1^2+ (B_1-B_2)\vert  +\vert B_1^2-(B_1-B_2)\vert),\\&\quad\quad\dfrac{1}{6}(\vert B_1^2-(B_1-B_2)\vert+B_1)\bigg\}, \intertext{and}
\vert a_4\vert \leq &\min \bigg\{\dfrac{1}{12}\bigg(B_1+2(\vert A_1\vert +\vert C_1\vert )\bigg),\\&\quad\quad\dfrac{1}{24}\bigg(\frac{B_1}{2}\bigg(\bigg\vert\frac{7B_1^2-4(B_2-B_1)}{B_1^2-2(B_2-B_1)}\bigg\vert+\bigg\vert\frac{3{B_1}^2+4(B_2-B_1) }{B_1^2-2(B_2-B_1)}\bigg\vert\bigg)\\&\quad\quad +4\bigg(\bigg\vert B_2-B_1+ \frac{B_1(4(B_1+B_3-2B_2)+5B_1^3)}{4(B_1^2+2(B_1-B_2))}\bigg\vert\bigg)\bigg)\bigg\},
\end{align*}
\end{itemize}
 where \[A_1= \dfrac{4 B_1^4-B_1^3+B_1^2 B_2-6 B_1^2+12 B_1 B_2+2 B_1 B_3-8 B_2^2}{ 4 \left(B_1^2+2 B_1-2 B_2\right)}, \] and
\[C_1=\dfrac{B_1 \left(B_1^3-3 B_1^2+3 B_1 B_2+2 B_1-4 B_2+2 B_3\right)}{4 \left(B_1^2+2 B_1-2 B_2\right)}.\]
\end{cor}

%\begin{rmk}
%For $f\in CV_{\sigma}(\rho)$, the class of bi-convex functions of order $\rho$, $(0\leq \rho < 1)$,
 %$B_1=2(1-\rho),B_2=2(1-\rho)$ and $B_3=2(1-\rho)$. Then from (\ref{64}), (\ref{65}) and (\ref{66}) we get
%\begin{equation*}
%\vert a_2\vert \leq (1-\rho),
%\end{equation*}
%\begin{equation*}
%\vert a_3\vert\leq \frac{1}{3} (1-\rho ) (3-2 \rho ),
%\end{equation*}
%\begin{equation*}
%\vert a_4\vert \leq \frac{1}{6} (1-\rho ) (6-5\rho).
%\end{equation*}
%\end{rmk}

\section{Estimate for the fifth coefficient}
Here we consider the class $ST_{\sigma}(\rho)$, $(0\leq \rho<1)$ of bi-starlike functions of order $\rho$. Note that from Remark \ref{74}, we have
\begin{equation*}
\vert a_2\vert \leq
\begin{cases}
\sqrt{2(1-\rho)}  \quad \quad (0\leq\rho\leq1/2)\\
2(1-\rho) \quad\quad \quad (1/2\leq \rho<1),
\end{cases}
\end{equation*}
\begin{equation*}
\vert a_3\vert \leq
\begin{cases}
2(1-\rho) \quad\quad\quad\quad \quad (0\leq \rho\leq 1/2)\\
(1-\rho)(3-2\rho)  \quad \quad (1/2\leq\rho<1),
\end{cases}
\end{equation*}
and
\begin{equation*}
\vert a_4\vert\leq
\begin{cases}
 \dfrac{2(1-\rho)}{3}(1+2\sqrt{2(1-\rho)})\quad  \quad (0\leq \rho \leq 1/2),\\
\dfrac{2(1-\rho)}{3}(1+4(1-\rho))\quad \quad\quad \quad (1/2\leq \rho <1).
\end{cases}
\end{equation*}
It is well known that for $f\in ST(\rho)$, $\vert a_n\vert \leq \dfrac{1}{(n-1)!}\prod _{k=2}^{n}(k-2\rho)$ for $n\geq 2 $. In particular, for $n=5$, $|a_5|\leq \dfrac{1}{24}(2-2\rho)(3-2\rho)(4-2\rho)(5-2\rho)$. For $f\in  ST_{\sigma}(\rho)$, we expect a smaller bound for the fifth coefficient. The bound is given by the following theorem.
 \begin{thm}
 Let $f(z)\in ST_{\sigma}(\rho)$. For $0\leq \rho \leq 1/2$, we have
\[ \vert a_5\vert \leq \frac{2}{3}(1-\rho)+\frac{3}{2}(1-\rho)^2+\frac{2}{3}\sqrt{2}(1-\rho)^{3/2}.\]
\end{thm}

\begin{proof}
 Let $f(z)\in ST_{\sigma}(\rho)$. Then by Definition \ref{75}, we get
\begin{equation}
\label{b}
\frac{zf'(z)}{f(z)}=\rho+(1-\rho)Q_{1}(z)
\end{equation}
  and
\begin{equation}
\label{c}
 \frac{w g'(w)}{g(w)}=\rho+(1-\rho)P_{1}(w)
\end{equation}
   respectively, where $Q_{1}$ and $P_{1}$ are analytic functions with positive real part in the unit disk $\mathbb{D}$. Let the functions $Q_1$ and $P_1$ have the series expansions \[Q_{1}(z)=1+c_1z+c_2z^2+\cdots\quad (z\in \mathbb{D})\] and \[P_{1}(w)=1+l_{1}w+l_{2}w^2+\cdots \quad (w\in \mathbb{D}).\]
Comparing the coefficients in equations (\ref{b}) and (\ref{c}), we get
\begin{gather}
a_2=(1-\rho)c_1,\label{d}\\
2a_3-a_2^2=(1-\rho)c_2,\label{e}\\
3a_4-3a_2a_3+a_2^3=(1-\rho)c_3\label{f}\\
4a_5-a_2^4+4a_2^2a_3-2a_3^2-4a_2a_4=(1-\rho)c_4 \label{g}
\end{gather}
and
\begin{gather}
-a_2=(1-\rho)l_1 \label{h}\\
3a_2^2-2a_3=(1-\rho)l_2\label{i}\\
-(10a_2^3-12a_2a_3+3a_4)=(1-\rho)l_3\label{jjj}\\
-4a_5+35a_2^4+10a_3^2+20a_2a_4-60a_2^2a_3=(1-\rho)l_4.\label{k}
\end{gather}
Addition of  (\ref{e}) and (\ref{i}) yields:
\begin{equation}
2a_2^2=(1-\rho)(c_2+l_2).
\end{equation}
Using (\ref{d}), we have
\begin{equation}\label{47}
 c_1^2=\frac{c_2+l_2}{2(1-\rho)}.
 \end{equation}
  By applying the inequalities $\vert c_2\vert \leq 2$ and $\vert l_2\vert\leq 2$ in (\ref{47}), we get \[\vert c_1\vert\leq\sqrt{\dfrac{2}{(1-\rho)}}  .\]
Using equations (\ref{d}), (\ref{e}) and (\ref{i}), we get
\begin{equation}
\label{bb}
a_3=(1-\rho)^2c_1^2+\frac{(1-\rho)}{4}(c_2-l_2).
\end{equation}
From  equations (\ref{f}), (\ref{jjj}) and (\ref{47}), we have
\begin{equation}
\label{aaa}
6a_4=2(1-\rho)c_3+\frac{7(1-\rho)^2}{2}c_1c_2+\frac{(1-\rho)^2}{2}c_1l_2.
\end{equation}
From equation (\ref{g}), we have
 \[4 a_5= (1-\rho) c_4+a_2^4-4a_2^2a_3+2a_3^2+4a_2a_4.\]
 Now substituting the values of $a_2$, $a_3$ and $a_4$ from (\ref{d}), (\ref{bb}) and (\ref{aaa}) respectively, we have \\
\begin{align*}   4a_5&= (1-\rho)c_4+((1-\rho)^2c_1^2+\frac{(1-\rho)}{4}(c_2-l_2))(1-\rho)c_2+\frac{1}{6}c_1(1-\rho)(2(1-\rho)c_3\\
&\quad{}+\frac{7(1-\rho)^2}{2}c_1c_2+\frac{(1-\rho)^2}{2}c_1l_2)+(1-\rho)^2c_1c_3.\end{align*}
On simplification, we get
\[4a_5= (1-\rho)c_4+\frac{19}{12} c_1^2c_2(1-\rho)^3+\frac{4}{3}c_1c_3(1-\rho)^2+\frac{1}{12}(1-\rho)^3c_1^2l_2+\frac{1}{4}(1-\rho)^2c_2(c_2-l_2).\]
Using equation (\ref{47}) in the above expression, we get
\[4a_5=(1-\rho)c_4+\frac{25}{24}(1-\rho)^2c_2^2+\frac{7}{12}(1-\rho)^2c_2l_2+\frac{1}{24}l_2^2(1-\rho)+\frac{4}{3}(1-\rho)^2c_1c_3.\]
Applying the inequalities $\vert c_2\vert\leq 2$, $\vert c_3\vert \leq 2$ and $\vert l_2\vert \leq 2$, we have
\[4\vert a_5\vert \leq 2(1-\rho)+\frac{20}{3}(1-\rho)^2+\frac{8}{3}(1-\rho)^2\vert c_1\vert.\]
Finally using estimate for $\vert c_1\vert $, we get
\[\vert a_5\vert\leq \frac{1}{2}(1-\rho)+\frac{5}{3}(1-\rho)^2+\frac{2\sqrt{2}}{3}(1-\rho)^{3/2}.\]
\end{proof}
\begin{cor}
Let $f(z)=z+\sum_{n=2}^{\infty}a_n z^n$, $(z\in \mathbb{D})$ be bi-starlike function. Then $\vert a_5\vert \leq \dfrac{13}{6}+\dfrac{2\sqrt{2}}{3}\simeq 3.10947. $
\end{cor}
Now we consider the class $SS_{\sigma}(\beta)$ of strongly bi-starlike fubctions of order $\beta$.
Ali and Singh \cite{rosi} proved that if $f$ is srongly starlike function of order $\beta $, then
\begin{equation}\label{67}
\vert a_5\vert \leq
\begin{cases}
\dfrac{\beta^2}{9}(38\beta^2+7)\quad\quad 38\beta^3-30\alpha^2+16\alpha\geq \dfrac{9}{2},\\
\dfrac{\beta}{2} \quad\quad\quad\quad\quad\quad 228\beta^4-194\beta^3+2\beta^2+39\beta-9\leq 0.
\end{cases}
\end{equation}

\noindent We prove the following:
\begin{thm}
Let the function $f\in SS_\sigma(\beta)$. Then for $1/2\leq \beta\leq 1$,
\[ \vert a_5\vert\leq \frac{\beta}{9} \left(30 \beta ^2-21 \beta+9 +\frac{ \left(38 \beta ^2-30 \beta +7\right) \beta }{(\beta +1)^4}+\frac{3 (7 \beta -3)}{\sqrt{\beta +1}}\right).\]
\end{thm}
\begin{proof}
Let $f(z)\in SS_\sigma(\beta)$. Then by Definition \ref{75}, we have
\begin{equation}
\label{1}
\frac{zf'(z)}{f(z)}=[Q(z)]^\beta
\end{equation}
  and
\begin{equation}
\label{2}
 \frac{w g'(w)}{g(w)}=[P(w)]^\beta
\end{equation}
   respectively, where $P$ and $Q$ are analytic functions with positive real part in the unit disk $\mathbb{D}$. Let \[Q(z)=1+c_1z+c_2z^2+\cdots\quad (z\in \mathbb{D})\] and \[P(w)=1+l_{1}w+l_{2}w^2+\cdots \quad (w\in \mathbb{D}).\]
Comparing the coefficients of both sides in equation (\ref{1}), we get
\begin{gather}
a_2=\beta c_1\label{3}\\
2a_3-a_2^2=\beta c_2+\frac{\beta(\beta-1)}{2}c_1^2\label{4}\\
3a_4-3a_2a_3+a_2^3=\beta c_3+\beta(\beta-1)c_1c_2+\frac{\beta(\beta-1)(\beta-2)}{6}c_1^3\label{5}\\
4a_5-a_2^4+4a_2^2a_3-2a_3^2-4a_2a_4=\beta c_4+\frac{\beta(\beta-1)}{2}(c_2^2+2c_1c_3)+\frac{\beta(\beta-1)(\beta-2)}{6}3c_1^2c_2\nonumber+\\\quad\quad\quad\frac{\beta(\beta-1)(\beta-2)(\beta-3)}{24}c_1^4\label{6}.
\end{gather}
Comparing the coefficients of both sides in equation (\ref{2}), we get
\begin{gather}
-a_2=\beta l_1\label{7}\\
3a_2^2-2a_3=\beta l_2+\frac{\beta(\beta-1)}{2}l_1^2\label{8}\\
-(10a_2^3-12a_2a_3+3a_4)=\beta l_3+\beta\beta-1)l_1l_2+\frac{\beta(\beta-1)(\beta-2)}{6}l_1^3\label{9}\\
-4a_5+35a_2^4+10a_3^2+20a_2a_4-60a_2^2a_3=\beta l_4+\frac{\beta(\beta-1)}{2}(l_2^2+2l_1l_3)+\frac{\beta(\beta-1)(\beta-2)}{6}3l_1^2l_2\nonumber+\\\quad\quad\quad\quad\quad\quad\frac{\beta(\beta-1)(\beta-2)(\beta-3)}{24}l_1^4\label{10}.
\end{gather}
Equations (\ref{3}) and (\ref{7}) yields $c_1=-l_1$.\\
Adding equations (\ref{4}) and (\ref{8}) and using $c_1=-l_1$, we get \[2a_2^2=\beta(c_2+l_2)+\beta(\beta-1)c_1^2.\] Substituting the value of $a_2$ from  equation (\ref{3}), we have
\begin{equation}\label{15}
c_1^2=\frac{c_2+l_2}{1+\beta}.
\end{equation}
 Using the well known inequalities $\vert c_2\vert \leq 2$ and $\vert l_2\vert \leq 2$, we have
 \begin{equation}\label{16}
 \vert c_1\vert \leq \frac{2}{\sqrt{1+\beta}}.
\end{equation}
Also, subtracting equation (\ref{8}) from (\ref{4}) and using $c_1=-l_1$ with equation (\ref{3}), we have
\begin{equation}\label{11}
a_3=\beta^2c_1^2+\frac{1}{4}\beta(c_2-l_2).
\end{equation}
Adding equations (\ref{5}) and (\ref{9}) and using $l_1=-c_1$, we have
\begin{equation}\label{12}
-9a_2^3+9a_2a_3=\beta(c_3+l_3)+\beta(\beta-1)c_1(c_2-l_2).
\end{equation}
Subtracting equation (\ref{9}) from equation (\ref{5}), we have
\begin{equation*}
\begin{split}
6a_4& =-11a_2^3+15a_2a_3+\beta(c_3-l_3)+\beta(\beta-1)c_1(c_2+l_2)+\frac{\beta(\beta-1)(\beta-2)}{3}c_1^3\\
& =-9a_2^3+9a_2a_3-2a_2^3+6a_2a_3+\beta(c_3-l_3)+\beta(\beta-1)c_1(c_2+l_2)+\frac{\beta(\beta-1)(\beta-2)}{3}c_1^3.
 \end{split}
\end{equation*}
Using equations (\ref{11}) and (\ref{12}), we have
\begin{equation*}
\begin{split}
6a_4 &=\beta(c_3+l_3)+\beta(\beta-1)c_1(c_2-l_2)-2\beta^3c_1^3+6\beta c_1\bigg(\beta^2c_1^2+\dfrac{1}{4}\beta(c_2-l_2)\bigg)\\&\quad \quad+\beta(c_3-l_3)+\beta(\beta-1)c_1(c_2+l_2)+\dfrac{\beta(\beta-1)(\beta-2)}{3}c_1^3\\
&=2\beta c_3+\dfrac{\beta(5\beta-2)}{2}c_1(c_2-l_2)+\beta(\beta-1)c_1(c_2+l_2)+\dfrac{13\beta^3-3\beta^2+2\beta}{3}c_1^3.
\end{split}
\end{equation*}
On replacing $c_1^2$ by $(c_2+l_2)/(1+\beta)$, we finally have
\begin{align}
6a_4&=\beta\bigg(2c_3+\frac{(5\beta-2)}{2}c_1(c_2-l_2)+\frac{16\beta^2-3\beta-1}{3(1+\beta)}c_1(c_2+l_2)\bigg)\notag\\
&=\beta\bigg(2c_3+\frac{47 \beta ^2+3 \beta -8}{6 (\beta +1)}c_1c_2+\frac{17 \beta ^2-15 \beta +4}{6 (\beta +1)}c_1l_2\bigg)\label{76}.
\end{align}
Now to find an estimate for $\vert a_5\vert $, we first express $a_5$ in terms of first four coefficients of $f(z)$. For this we subtract equation (\ref{10}) from equation (\ref{6}) and use the fact that $c_1=-l_1$ and get
\begin{eqnarray}
8a_5&= 36a_2^4-64a_2^2a_3+12a_3^2+24a_2a_4+\beta(c_4-l_4)+\dfrac{\beta(\beta-1)}{2}(c_2^2-l_2^2+2c_1(c_3+l_3))\nonumber\\&\quad \quad+\dfrac{\beta(\beta-1)(\beta-2)}{2}c_1^2(c_2-l_2)\nonumber\\
&=34a_2^4+8a_3^2+16a_2a_4-56a_2^2a_3+2a_2^4-8a_2^2a_3+4a_3^2+8a_2a_4+\beta(c_4-l_4)\nonumber\\&\quad\quad +\dfrac{\beta(\beta-1)}{2}(c_2^2-l_2^2+2c_1(c_3+l_3))+\dfrac{\beta(\beta-1)(\beta-2)}{2}c_1^2(c_2-l_2)\label{13}.
\end{eqnarray}
Also, adding equations (\ref{6}) and (\ref{10}), we have
\begin{equation}\label{14}
\begin{split}
34a_2^4+8a_3^2+16a_2a_4-56a_2^2a_3&=\beta(c_4+l_4)+\frac{\beta(\beta-1)}{2}(c_2^2+l_2^2+2c_1(c_3-l_3))\\&\quad\quad+\frac{\beta(\beta-1)(\beta-2)}{2}c_1^2(c_2+l_2)+\frac{\beta(\beta-1)(\beta-2)(\beta-3)}{12}c_1^4.
\end{split}
\end{equation}
Using equation (\ref{14}) in (\ref{13}), we have
\begin{equation*}
\begin{split}
8a_5&=2\beta c_4+\beta(\beta-1)(c_2^2+2c_1c_3)+\beta(\beta-1)(\beta-2)c_1^2c_2+\frac{\beta(\beta-1)(\beta-2)(\beta-3)}{12}c_1^4\\&\quad\quad+2a_2^4-8a_2^2a_3+4a_3^2+8a_2a_4.
\end{split}
\end{equation*}
This gives
\begin{equation*}
\begin{split}
4a_5&=\beta c_4+\frac{\beta(\beta-1)}{2}(c_2^2+2c_1c_3)+\frac{\beta(\beta-1)(\beta-2)}{2}c_1^2c_2+\frac{\beta(\beta-1)(\beta-2)(\beta-3)}{24}c_1^4\\&\quad\quad+a_2^4-4a_2^2a_3+2a_3^2+4a_2a_4.\\
&=\beta c_4+\frac{\beta(\beta-1)}{2}(c_2^2+2c_1c_3)+\frac{\beta(\beta-1)(\beta-2)}{2}c_1^2c_2+\frac{\beta(\beta-1)(\beta-2)(\beta-3)}{24}c_1^4\\&\quad\quad+a_3(2a_3-a_2^2)+a_2a_4+a_2(a_2^3-3a_2a_3+3a_4)\\
&=\beta c_4+\frac{\beta(\beta-1)}{2}(c_2^2+2c_1c_3)+\frac{\beta(\beta-1)(\beta-2)}{2}c_1^2c_2+\frac{\beta(\beta-1)(\beta-2)(\beta-3)}{24}c_1^4\\&\quad\quad+ (\beta^2c_1^2+\frac{1}{4}\beta(c_2-l_2))(\beta c_2+\frac{\beta(\beta-1)}{2}c_1^2)+\frac{1}{6}\beta^2 c_1\bigg(2c_3+\frac{47 \beta ^2+3 \beta -8}{6 (\beta +1)}c_1c_2\\&\quad+\frac{17 \beta ^2-15 \beta +4}{6 (\beta +1)}c_1l_2\bigg)+\beta c_1\bigg(\beta c_3+\beta(\beta-1)c_1c_2+\frac{\beta(\beta-1)(\beta-2)}{6}c_1^3\bigg).
\end{split}
\end{equation*}
After simplification, we get
\begin{equation*}
\begin{split}
4a_5&= \beta c_4+\frac{\beta(3\beta-2)}{4}c_2^2+\frac{\beta(7\beta-3)}{3}c_1c_3+\bigg(\frac{\beta(283\beta^3+6 \beta^2-133\beta+72)}{72(1+\beta)}\bigg)c_1^2c_2\\&\quad\quad+\frac{\beta^2(25\beta^2-30\beta+17)}{72(1+\beta)}c_1^2l_2-\frac{\beta^2}{4}c_2l_2+\frac{\beta(\beta-1)(17\beta^2-13\beta+6)}{24}c_1^4.
\end{split}
\end{equation*}
Substituting $l_2=(1+\beta)c_1^2-c_2$ from equation (\ref{15}), the above equation reduces to
\begin{equation*}
\begin{split}
4a_5 &=\beta c_4+\frac{\beta(2\beta-1)}{2}c_2^2+\frac{\beta(7\beta-3)}{3}c_1c_3+\frac{1}{3} \beta  \left(10 \beta ^2-10 \beta +3\right)c_1^2c_2\\&\quad+\frac{1}{36} \beta  \left(38 \beta ^3-60 \beta ^2+37 \beta -9\right)c_1^4\\
&= \beta c_4+\frac{\beta(2\beta-1)}{2}c_2^2+\frac{\beta(7\beta-3)}{3}c_1c_3+\frac{ \beta  \left(10 \beta ^2-10 \beta +3\right)}{12}(4c_1^2c_2-c_1^4)\\&\quad+\frac{\beta  \left(38 \beta ^3-30 \beta ^2+7 \beta \right)}{36} c_1^4.
\end{split}
\end{equation*}
For $1/2\leq \beta<1$, using the estimate for $\vert c_1\vert$ given by (\ref{16}) and the inequality  $\vert 4c_1^2c_2-c_1^4\vert\leq 16$, we get
\[ \vert a_5\vert\leq \frac{\beta}{9} \left(30 \beta ^2-21 \beta+9 +\frac{ \left(38 \beta ^2-30 \beta +7\right) \beta }{(\beta +1)^4}+\frac{3 (7 \beta -3)}{\sqrt{\beta +1}}\right).\qedhere\]
\end{proof}
Note that from (\ref{67}), for $f\in SS(1/2)$, $\vert a_5\vert\leq \frac{33}{72}\simeq 0.45833.$  We have the following:
\begin{cor}
Let $f(z)$ given by (\ref{aa}) be in the class $SS_{\sigma}(1/2)$. Then \[\vert a_5\vert \leq
 \frac{1}{36} \left(\frac{332}{27}+\sqrt{6}\right)\simeq 0.409605.\]
\end{cor}

 \end {document}